\documentclass{amsart}
\usepackage{amsmath}
\usepackage{amssymb}
\usepackage{amsfonts}

\newtheorem{theorem}{Theorem}
\theoremstyle{plain}

\newtheorem{corollary}{Corollary}

\newtheorem{definition}{Definition}

\newtheorem{lemma}{Lemma}

\newtheorem{remark}{Remark}

\numberwithin{equation}{section}
\input{tcilatex}

\begin{document}
\title[On new inequalities]{On new inequalities via Riemann-Liouville
fractional integration }
\author{Mehmet Zeki Sar\i kaya$^{\star \clubsuit }$}
\address{$^{\clubsuit }$Department of Mathematics,Faculty of Science and
Arts, D\"{u}zce University, D\"{u}zce, Turkey}
\email{sarikayamz@gmail.com}
\thanks{$^{\star }$corresponding author}
\author{Hasan Ogunmez$^{\blacklozenge }$}
\address{$^{\blacklozenge }$Department of Mathematics, \ Faculty of Science
and Arts, Afyon Kocatepe University, Afyon-TURKEY}
\email{hogunmez@aku.edu.tr}
\date{}

\begin{abstract}
In this paper, we extend the Montogomery identities for the
Riemann-Liouville fractional integrals. We also use this Montogomery
identities to establish some new integral inequalities for convex functions.
\end{abstract}

\maketitle

\section{Introduction}

The inequality of Ostrowski \cite{Ostrowski} gives us an estimate for the
deviation of the values of a smooth function from its mean value. More
precisely, if $f:[a,b]\rightarrow \mathbb{R}$ is a differentiable function
with bounded derivative, then%
\begin{equation*}
\left\vert f(x)-\frac{1}{b-a}\int\limits_{a}^{b}f(t)dt\right\vert \leq \left[
\frac{1}{4}+\frac{(x-\frac{a+b}{2})^{2}}{(b-a)^{2}}\right] (b-a)\left\Vert
f^{\prime }\right\Vert _{\infty }
\end{equation*}
for every $x\in \lbrack a,b]$. Moreover the constant $1/4$ is the best
possible.

For some generalizations of this classic fact see the book \cite[p.468-484]%
{mitrovich} by Mitrinovic Pecaric and Fink. A simple proof of this fact can
be \ done by using the following identity \cite{mitrovich}:

If $f:[a,b]\rightarrow \mathbb{R}$ is differentiable on $[a,b]$ with the
first derivative $f^{\prime }$ integrable on $[a,b],$ then Montgomery
identity holds:%
\begin{equation*}
f(x)=\frac{1}{b-a}\int\limits_{a}^{b}f(t)dt+\int\limits_{a}^{b}P_{1}(x,t)f^{%
\prime }(t)dt,
\end{equation*}%
where $P_{1}(x,t)$ is the Peano kernel defined by%
\begin{equation*}
P_{1}(x,t):=\left\{ 
\begin{array}{ll}
\dfrac{t-a}{b-a}, & a\leq t<x \\ 
&  \\ 
\dfrac{t-b}{b-a}, & x\leq t\leq b.%
\end{array}%
\right. 
\end{equation*}%
Recently, several generalizations of the Ostrowski integral inequality are
considered by many authors; for instance covering the following concepts:
functions of bounded variation, Lipschitzian, monotonic, absolutely
continuous and $n$-times differentiable mappings with error estimates with
some special means together with some numerical quadrature rules. For recent
results and generalizations concerning Ostrowski's inequality, we refer the
reader to the recent papers  \cite{Cerone1}, \cite{Duo}, \cite{Dragomir}-%
\cite{Liu}, \cite{sarikaya}-\cite{sarikaya2}.

In this article, we use the Riemann-Liouville fractional integrals to
establish some new integral inequalities of Ostrowski's type. From our
results, the weighted and the classical Ostrowski's inequalities can be
deduced as some special cases.

\section{Fractional Calculus}

Firstly, we give some necessary definitions and mathematical preliminaries
of fractional calculus theory which are used further in this paper. More
details, one can consult \cite{gorenflo}, \cite{samko}.

\begin{definition}
The Riemann-Liouville fractional integral operator of order $\alpha \geq 0$
with $a\geq 0$ is defined as%
\begin{eqnarray*}
J_{a}^{\alpha }f(x) &=&\frac{1}{\Gamma (\alpha )}\dint\limits_{a}^{x}(x-t)^{%
\alpha -1}f(t)dt, \\
J_{a}^{0}f(x) &=&f(x).
\end{eqnarray*}
\end{definition}

Recently, many authors have studied a number of inequalities by used the
Riemann-Liouville fractional integrals, see (\cite{Anastassiou}, \cite%
{Belarbi}, \cite{Dahmani}, \cite{Dahmani1}) and the references cited therein.

\section{Main Results}

In order to prove our main results, we need the following identities, which
corrects the result proved in \cite{Anastassiou}:

\begin{lemma}
\label{lm} Let $f:I\subset \mathbb{R}\rightarrow \mathbb{R}$ be
differentiable function on $I^{\circ }$ with $a,b\in I$ ($a<b$) and $%
f^{\prime }\in L_{1}[a,b]$, then%
\begin{equation}
f(x)=\frac{\Gamma (\alpha )}{b-a}(b-x)^{1-\alpha }{\Large J}_{a}^{\alpha
}f(b)-{\Large J}_{a}^{\alpha -1}(P_{2}(x,b)f(b))+{\Large J}_{a}^{\alpha
}(P_{2}(x,b)f^{^{\prime }}(b)),\ \ \ \alpha \geq 1,  \label{1}
\end{equation}%
where $P_{2}(x,t)$ is the fractional Peano kernel defined by%
\begin{equation}
P_{2}(x,t):=\left\{ 
\begin{array}{ll}
\dfrac{t-a}{b-a}(b-x)^{1-\alpha }\Gamma (\alpha ), & a\leq t<x \\ 
&  \\ 
\dfrac{t-b}{b-a}(b-x)^{1-\alpha }\Gamma (\alpha ), & x\leq t\leq b.%
\end{array}%
\right.  \label{3}
\end{equation}
\end{lemma}

\begin{proof}
By definition of $P_{1}(x,t)$, we have%
\begin{eqnarray*}
\Gamma (\alpha ){\Large J}_{a}^{\alpha }(P_{1}(x,b)f^{^{\prime }}(b))
&=&\int\limits_{a}^{b}(b-t)^{\alpha -1}P_{1}(x,t)f^{^{\prime }}(t)dt \\
&& \\
&=&\int\limits_{a}^{x}(b-t)^{\alpha -1}\left( \frac{t-a}{b-a}\right)
f^{^{\prime }}(t)dt+\int\limits_{x}^{b}(b-t)^{\alpha -1}\left( \frac{t-b}{b-a%
}\right) f^{^{\prime }}(t)dt \\
&& \\
&=&\frac{1}{b-a}\left[ \int\limits_{a}^{x}(b-t)^{\alpha -1}(t-a)f^{^{\prime
}}(t)dt-\int\limits_{x}^{b}(b-t)^{\alpha }f^{^{\prime }}(t)dt\right] \\
&& \\
&=&\frac{1}{b-a}(I_{1}+I_{2}).
\end{eqnarray*}%
Integrating by parts, we can state:%
\begin{eqnarray}
I_{1} &=&\left. (b-t)^{\alpha -1}(t-a)f(t)\right\vert
_{a}^{x}-\int\limits_{a}^{x}[-(\alpha -1)(b-t)^{\alpha
-2}(t-a)+(b-t)^{\alpha -1}]~f(t)dt  \notag \\
&&  \label{4} \\
&=&(b-x)^{\alpha -1}(x-a)f(x)+(\alpha -1)\int\limits_{a}^{x}(b-t)^{\alpha
-2}(t-a)~f(t)dt-\int\limits_{a}^{x}(b-t)^{\alpha -1}~f(t)dt  \notag
\end{eqnarray}%
and similary,%
\begin{equation}
I_{2}=\left. -(b-t)^{\alpha }f(t)\right\vert _{x}^{b}-\alpha
\int\limits_{x}^{b}(b-t)^{\alpha -1}~f(t)dt=(b-x)^{\alpha }f(x)-\alpha
\int\limits_{x}^{b}(b-t)^{\alpha -1}~f(t)dt  \label{5}
\end{equation}%
Adding (\ref{4}) and (\ref{5}), we get%
\begin{eqnarray*}
\Gamma (\alpha )J_{a}^{\alpha }(P_{1}(x,b)f^{^{\prime }}(b)) &=&\frac{1}{b-a}%
\left\{ (b-a)(b-x)^{\alpha -1}f(x)+(\alpha
-1)\int\limits_{a}^{x}(b-t)^{\alpha -2}(t-a)f(t)dt\right. \\
&& \\
&&\left. -\alpha \int\limits_{x}^{b}(b-t)^{\alpha
-1}f(t)dt-\int\limits_{a}^{x}(b-t)^{\alpha -1}f(t)dt\right\} .
\end{eqnarray*}%
If we add and subtract the integral $(\alpha
-1)\int\limits_{x}^{b}(b-t)^{\alpha -2}(t-b)f(t)dt$ to the right hand side
of the equation above, then we have 
\begin{eqnarray*}
\Gamma (\alpha )J_{a}^{\alpha }(P_{1}(x,b)f^{^{\prime }}(b)) &=&\frac{1}{b-a}%
\left\{ (b-a)(b-x)^{\alpha -1}f(x)\right. \\
&& \\
&&\left. +(b-a)(\alpha -1)\int\limits_{a}^{b}(b-t)^{\alpha
-2}P_{1}(x,t)f(t)dt-\int\limits_{a}^{b}(b-t)^{\alpha -1}f(t)dt\right\} \\
&& \\
&=&(b-x)^{\alpha -1}f(x)+(\alpha -1)\int\limits_{a}^{b}(b-t)^{\alpha
-2}P_{1}(x,t)f(t)dt-\frac{1}{b-a}\int\limits_{a}^{b}(b-t)^{\alpha -1}f(t)dt
\\
&& \\
&=&(b-x)^{\alpha -1}f(x)-\frac{\Gamma (\alpha )}{b-a}{\Large J}_{a}^{\alpha
}f(b)+\Gamma (\alpha ){\Large J}_{a}^{\alpha -1}(P_{1}(x,b)f(b)).
\end{eqnarray*}%
Multiplying the both sides by $(b-x)^{1-\alpha }$, we obtain%
\begin{equation*}
{\Large J}_{a}^{\alpha }(P_{2}(x,b)f^{^{\prime }}(b))=f(x)-\frac{\Gamma
(\alpha )}{b-a}(b-x)^{1-\alpha }{\Large J}_{a}^{\alpha }f(b)+{\Large J}%
_{a}^{\alpha -1}(P_{2}(x,b)f(b))
\end{equation*}%
and so%
\begin{equation*}
f(x)=\frac{\Gamma (\alpha )}{b-a}(b-x)^{1-\alpha }{\Large J}_{a}^{\alpha
}f(b)-{\Large J}_{a}^{\alpha -1}(P_{2}(x,b)f(b))+{\Large J}_{a}^{\alpha
}(P_{2}(x,b)f^{^{\prime }}(b)).
\end{equation*}%
This completes the proof.
\end{proof}

\begin{remark}
If we choose $\alpha =1$, the formula (\ref{1}) reduces to the classical
Montgomery Identity with $P_{2}(x,b)f(b)=0,~x\leq t\leq b$. However, because
of the paces of the $<$ and $\leq $ signs in \cite{Anastassiou}, the
corresp\i nding result therein does not reduce to the classical Montgomery
Identity.
\end{remark}

\begin{theorem}
\label{thm1} Let $f:[a,b]\rightarrow \mathbb{R}$ be a convex function on $%
[a,b].$ Then for any $x\in (a,b),$ the following inequality holds:%
\begin{eqnarray}
&&\frac{1}{\alpha (\alpha +1)}\left[ \alpha \frac{(b-x)^{2}}{b-a}%
f_{+}^{\prime }(x)-\left( (b-a)^{\alpha }(b-x)^{1-\alpha }+\alpha \frac{%
(b-x)^{2}}{b-a}-(\alpha +1)(b-x)\right) f_{-}^{\prime }(x)\right]  \notag \\
&&  \label{11} \\
&\leq &\frac{\Gamma (\alpha )}{b-a}(b-x)^{1-\alpha }{\Large J}_{a}^{\alpha
}f(b)-{\Large J}_{a}^{\alpha -1}(P_{2}(x,b)f(b))-f(x),\ \alpha \geq 1. 
\notag
\end{eqnarray}
\end{theorem}

\begin{proof}
From Lemma \ref{lm}, we have%
\begin{eqnarray}
&&f(x)-\frac{\Gamma (\alpha )}{b-a}(b-x)^{1-\alpha }{\Large J}_{a}^{\alpha
}f(b)+{\Large J}_{a}^{\alpha -1}(P_{2}(x,b)f(b))  \notag \\
&&  \label{12} \\
&=&\frac{(b-x)^{1-\alpha }}{b-a}\left[ \int\limits_{a}^{x}(b-t)^{\alpha
-1}\left( t-a\right) f^{^{\prime }}(t)dt-\int\limits_{x}^{b}(b-t)^{\alpha
}f^{^{\prime }}(t)dt\right] .  \notag
\end{eqnarray}%
Since $f$ is convex, then for any $x\in (a,b)$ we have the following
inequalities%
\begin{equation}
f^{\prime }(t)\leq f_{-}^{\prime }(x)\text{ \ \ \ for \ a.e. }t\in \lbrack
a,x]  \label{13}
\end{equation}%
\begin{equation}
f^{\prime }(t)\geq f_{+}^{\prime }(x)\text{ \ \ \ for \ a.e. }t\in \lbrack
x,b].  \label{14}
\end{equation}%
If we multiply (\ref{13}) by $(b-t)^{\alpha -1}\left( t-a\right) \geq 0,\
t\in \lbrack a,x],\ \alpha \geq 1$ and integrate on $[a,x],$ we get%
\begin{eqnarray}
&&\int\limits_{a}^{x}(b-t)^{\alpha -1}\left( t-a\right) f^{^{\prime
}}(t)dt\leq \int\limits_{a}^{x}(b-t)^{\alpha -1}\left( t-a\right)
f_{-}^{^{\prime }}(t)dt  \notag \\
&&  \label{15} \\
&=&\frac{1}{\alpha (\alpha +1)}\left[ (b-a)^{\alpha +1}+(b-x)^{\alpha }\left[
\alpha (b-x)-(\alpha +1)\left( b-a\right) \right] \right] f_{-}^{\prime }(x)
\notag
\end{eqnarray}%
and if we multiply (\ref{14}) by $(b-t)^{\alpha }\geq 0,\ t\in \lbrack
x,b],\ \alpha \geq 1$ and integrate on $[x,b],$ we also get 
\begin{equation}
\int\limits_{x}^{b}(b-t)^{\alpha }f^{^{\prime }}(t)dt\geq
\int\limits_{x}^{b}(b-t)^{\alpha }f_{+}^{^{\prime }}(t)dt=\frac{%
(b-x)^{\alpha +1}}{\alpha +1}f_{+}^{\prime }(x).  \label{16}
\end{equation}%
Finally, if we subtract (\ref{16}) from (\ref{15}) and use the
representtation (\ref{12}) we deduce the desired inequality (\ref{11}).
\end{proof}

\begin{corollary}
\label{cr} Under the assumptions Theorem \ref{thm1} with $\alpha =1,$ we have%
\begin{equation*}
\frac{1}{2}\left[ (b-x)^{2}f_{+}^{\prime }(x)-(a-x)^{2}f_{-}^{\prime }(x)%
\right] \leq \int\limits_{a}^{b}f(t)dt-(b-a)f(x).
\end{equation*}
\end{corollary}

The proof of the Corrollary \ref{cr} is proved by Dragomir in \cite%
{Dragomir1}. Hence, our results in Theorem \ref{thm1} are generalizations of
the corresponding results of Dragomir \cite{Dragomir1}.

\begin{remark}
If we take $x=\frac{a+b}{2}$ in Corollary \ref{cr}, we get%
\begin{equation*}
0\leq \frac{b-a}{8}\left[ f_{+}^{\prime }(\frac{a+b}{2})-f_{-}^{\prime }(%
\frac{a+b}{2})\right] \leq \frac{1}{b-a}\int\limits_{a}^{b}f(t)dt-f(\frac{a+b%
}{2}).
\end{equation*}
\end{remark}

\begin{theorem}
\label{thm2} Let $f:[a,b]\rightarrow \mathbb{R}$ be a convex function on $%
[a,b].$ Then for any $x\in \lbrack a,b],$ the following inequality holds:%
\begin{eqnarray}
&&\frac{\Gamma (\alpha )}{b-a}(b-x)^{1-\alpha }{\Large J}_{a}^{\alpha }f(b)-%
{\Large J}_{a}^{\alpha -1}(P_{2}(x,b)f(b))-f(x)  \notag \\
&&  \label{17} \\
&\leq &\frac{1}{\alpha (\alpha +1)}\left[ \alpha \frac{(b-x)^{2}}{b-a}%
f_{-}^{\prime }(b)-\left( (b-a)^{\alpha }(b-x)^{1-\alpha }+\alpha \frac{%
(b-x)^{2}}{b-a}-(\alpha +1)(b-x)\right) f_{+}^{^{\prime }}(a)\right] ,\ \ \
\ \alpha \geq 1.  \notag
\end{eqnarray}
\end{theorem}

\begin{proof}
Assume that $f_{+}^{\prime }(a)$ and $f_{-}^{\prime }(b)$\ are finite. Since 
$f$ is convex on $[a,b]$, then we have the following inequalities%
\begin{equation}
f^{\prime }(t)\geq f_{+}^{\prime }(a)\text{ \ \ \ for \ a.e. }t\in \lbrack
a,x]  \label{18}
\end{equation}%
\begin{equation}
f^{\prime }(t)\leq f_{-}^{\prime }(b)\text{ \ \ \ for \ a.e. }t\in \lbrack
x,b].  \label{19}
\end{equation}%
If we multiply (\ref{18}) by $(b-t)^{\alpha -1}\left( t-a\right) \geq 0,\
t\in \lbrack a,x],\ \alpha \geq 1$ and integrate on $[a,x],$ we have%
\begin{eqnarray}
&&\int\limits_{a}^{x}(b-t)^{\alpha -1}\left( t-a\right) f^{^{\prime
}}(t)dt\geq \int\limits_{a}^{x}(b-t)^{\alpha -1}\left( t-a\right)
f_{+}^{^{\prime }}(a)dt  \notag \\
&&  \label{20} \\
&=&\frac{1}{\alpha (\alpha +1)}\left[ (b-a)^{\alpha +1}+(b-x)^{\alpha }\left[
\alpha (b-x)-(\alpha +1)\left( b-a\right) \right] \right] f_{+}^{^{\prime
}}(a)  \notag
\end{eqnarray}%
and if we multiply (\ref{19}) by $(b-t)^{\alpha }\geq 0,\ t\in \lbrack
x,b],\ \alpha \geq 1$ and integrate on $[x,b],$ we also have 
\begin{equation}
\int\limits_{x}^{b}(b-t)^{\alpha }f^{^{\prime }}(t)dt\leq
\int\limits_{x}^{b}(b-t)^{\alpha }f_{+}^{^{\prime }}(t)dt=\frac{%
(b-x)^{\alpha +1}}{\alpha +1}f_{-}^{\prime }(b).  \label{21}
\end{equation}%
Finally, if we subtract (\ref{20}) from (\ref{21}) and use the
representtation (\ref{12}) we deduce the desired inequality (\ref{17}).
\end{proof}

\begin{corollary}
\label{cr1} Under the assumptions Theorem \ref{thm2} with $\alpha =1,$ we
have%
\begin{equation*}
\int\limits_{a}^{b}f(t)dt-(b-a)f(x)\leq \frac{1}{2}\left[ (b-x)^{2}f_{-}^{%
\prime }(b)-(a-x)^{2}f_{+}^{\prime }(a)\right] .
\end{equation*}
\end{corollary}

The proof of the Corrollary \ref{cr1} is proved by Dragomir in \cite%
{Dragomir1}. So, our results in Theorem \ref{thm2} are generalizations of
the corresponding results of Dragomir \cite{Dragomir1}.

\begin{remark}
If we take $x=\frac{a+b}{2}$ in Corollary \ref{cr1}, we get%
\begin{equation*}
0\leq \frac{1}{b-a}\int\limits_{a}^{b}f(t)dt-f(\frac{a+b}{2})\leq \frac{b-a}{%
8}\left[ f_{-}^{\prime }(b)-f_{+}^{\prime }(a)\right] .
\end{equation*}
\end{remark}

Now, we extend the Lemma \ref{lm} as follows:

\begin{theorem}
\label{thm} Let $f:I\subset \mathbb{R}\rightarrow \mathbb{R}$ be twice
differentiable function on $I^{\circ }$ with $f^{\prime }\in L_{1}[a,b]$,
then the following identity holds:%
\begin{equation}
\begin{array}{l}
(1-2\lambda )f(x)=\frac{\Gamma (\alpha )}{b-a}(b-x)^{1-\alpha }{\Large J}%
_{a}^{\alpha }f(b)-\lambda \left( \frac{b-a}{b-x}\right) ^{\alpha -1}f(a) \\ 
\\ 
\ \ \ \ \ \ \ \ \ \ \ \ \ \ \ \ \ \ \ \ \ \ \ \ \ \ \ \ \ \ \ \ -{\Large J}%
_{a}^{\alpha -1}(P_{3}(x,b)f(b))+{\Large J}_{a}^{\alpha
}(P_{3}(x,b)f^{^{\prime }}(b)),\ \ \ \alpha \geq 1,%
\end{array}
\label{6}
\end{equation}%
where $P_{3}(x,t)$ is the fractional Peano kernel defined by%
\begin{equation}
P_{3}(x,t):=\left\{ 
\begin{array}{ll}
\dfrac{t-(1-\lambda )a-\lambda b}{b-a}(b-x)^{1-\alpha }\Gamma (\alpha ), & 
a\leq t<x \\ 
&  \\ 
\dfrac{t-(1-\lambda )b-\lambda a}{b-a}(b-x)^{1-\alpha }\Gamma (\alpha ), & 
x\leq t\leq b%
\end{array}%
\right.  \label{7}
\end{equation}%
for $0\leq \lambda \leq 1.$
\end{theorem}

\begin{proof}
By similar way in proof of the Lemma \ref{lm}, we have%
\begin{eqnarray}
\Gamma (\alpha ){\Large J}_{a}^{\alpha }(P_{3}(x,b)f^{^{\prime }}(b))
&=&\int\limits_{a}^{b}(b-t)^{\alpha -1}P_{3}(x,t)f^{^{\prime }}(t)dt  \notag
\\
&&  \notag \\
&=&\frac{\Gamma (\alpha )(b-x)^{1-\alpha }}{b-a}\left[ \int%
\limits_{a}^{x}(b-t)^{\alpha -1}\left( t-(1-\lambda )a-\lambda b\right)
f^{^{\prime }}(t)dt\right.  \notag \\
&&  \label{8} \\
&&\left. +\int\limits_{x}^{b}(b-t)^{\alpha -1}\left( t-(1-\lambda )b-\lambda
a\right) f^{^{\prime }}(t)dt\right]  \notag \\
&&  \notag \\
&=&\frac{\Gamma (\alpha )(b-x)^{1-\alpha }}{b-a}(J_{1}+J_{2}).  \notag
\end{eqnarray}%
Integrating by parts, we can state:%
\begin{eqnarray*}
J_{1} &=&(b-x)^{\alpha -1}(x-(1-\lambda )a-\lambda b)f(x)+(b-a)^{\alpha }f(a)
\\
&& \\
&&+(\alpha -1)\int\limits_{a}^{x}(b-t)^{\alpha -2}(t-(1-\lambda )a-\lambda
b)~f(t)dt-\int\limits_{a}^{x}(b-t)^{\alpha -1}~f(t)dt
\end{eqnarray*}%
and similary,%
\begin{eqnarray*}
J_{2} &=&-(b-x)^{\alpha }(x-(1-\lambda )b-\lambda a)f(x) \\
&& \\
&&+(\alpha -1)\int\limits_{x}^{b}(b-t)^{\alpha -2}(t-(1-\lambda )a-\lambda
b)~f(t)dt-\int\limits_{x}^{b}(b-t)^{\alpha -1}~f(t)dt.
\end{eqnarray*}%
Thus, using $J_{1}$ and $J_{2}$ in (\ref{8}) we get (\ref{6}) which
completes the proof.
\end{proof}

\begin{remark}
We note that in the special cases, if we take $\lambda =0$ in Theorem \ref%
{thm}, then we get (\ref{1}) with the kernel $P_{2}(x,t)$.
\end{remark}

\begin{theorem}
\label{thm3} Let $f:[a,b]\rightarrow \mathbb{R}$ be differentiable on $(a,b)$
such that $f^{^{\prime }}\in L_{1}[a,b],$ where $a<b.$ If $\left\vert
f^{^{\prime }}(x)\right\vert \leq M$ for every $x\in \lbrack a,b]$ and $%
\alpha \geq 1$, then the following inequalty holds:%
\begin{eqnarray}
&&\left\vert (1-2\lambda )f(x)-\frac{\Gamma (\alpha )}{b-a}(b-x)^{1-\alpha }%
{\Large J}_{a}^{\alpha }f(b)+\lambda \left( \frac{b-a}{b-x}\right) ^{\alpha
-1}f(a)+{\Large J}_{a}^{\alpha -1}(P_{3}(x,b)f(b))\right\vert  \notag \\
&&  \notag \\
&\leq &\frac{M}{\alpha \left( \alpha +1\right) }\left\{ (b-a)^{\alpha
}(b-x)^{1-\alpha }\left[ 2\lambda ^{\alpha +1}+2(1-\lambda )^{\alpha
+1}+\lambda (b-a)-1\right] \right.  \notag \\
&&  \label{9} \\
&&\left. +(b-x)\left[ 2\alpha \frac{b-x}{b-a}-\left( \alpha +1\right) \right]
\right\} .  \notag
\end{eqnarray}

\begin{proof}
From Theorem \ref{thm}, we get%
\begin{eqnarray}
&&\left\vert (1-2\lambda )f(x)=\frac{\Gamma (\alpha )}{b-a}(b-x)^{1-\alpha }%
{\Large J}_{a}^{\alpha }f(b)-\lambda \left( \frac{b-a}{b-x}\right) ^{\alpha
-1}f(a)\right\vert  \notag \\
&&  \notag \\
&\leq &\frac{1}{\Gamma (\alpha )}\int\limits_{a}^{b}(b-t)^{\alpha
-1}\left\vert P_{3}(x,t)\right\vert \left\vert f^{^{\prime }}(t)\right\vert
dt  \notag \\
&&  \notag \\
&=&\frac{(b-x)^{1-\alpha }}{b-a}\left[ \int\limits_{a}^{x}(b-t)^{\alpha
-1}\left\vert t-(1-\lambda )a-\lambda b\right\vert \left\vert f^{^{\prime
}}(t)\right\vert dt\right.  \notag \\
&&  \label{10} \\
&&\left. +\int\limits_{x}^{b}(b-t)^{\alpha -1}\left\vert t-(1-\lambda
)b-\lambda a\right\vert \left\vert f^{^{\prime }}(t)\right\vert dt\right] 
\notag \\
&&  \notag \\
&\leq &\frac{M(b-x)^{1-\alpha }}{b-a}\left\{
\int\limits_{a}^{x}(b-t)^{\alpha -1}\left\vert t-(1-\lambda )a-\lambda
b\right\vert dt+\int\limits_{x}^{b}(b-t)^{\alpha -1}\left\vert t-(1-\lambda
)b-\lambda a\right\vert dt\right\}  \notag \\
&&  \notag \\
&=&\frac{M(b-x)^{1-\alpha }}{b-a}\left\{ J_{3}+J_{4}\right\} .  \notag
\end{eqnarray}%
By simple computation, we obtain%
\begin{eqnarray*}
J_{3} &=&\int\limits_{a}^{x}(b-t)^{\alpha -1}\left\vert t-(1-\lambda
)a-\lambda b\right\vert dt \\
&& \\
&=&\int\limits_{a}^{\lambda b+(1-\lambda )a}(b-t)^{\alpha -1}\left( \lambda
b+(1-\lambda )a-t\right) dt+\int\limits_{\lambda b+(1-\lambda
)a}^{x}(b-t)^{\alpha -1}\left( t-\lambda b-(1-\lambda )a\right) dt \\
&& \\
&=&\frac{(b-a)^{\alpha +1}}{\alpha \left( \alpha +1\right) }\left[
2(1-\lambda )^{\alpha +1}+\lambda (b-a)-1\right] +\frac{(b-x)^{\alpha }}{%
\alpha \left( \alpha +1\right) }\left[ \alpha (b-x)-(1-\lambda )(b-a)\left(
\alpha +1\right) \right]
\end{eqnarray*}%
and similarly%
\begin{eqnarray*}
J_{4} &=&\int\limits_{x}^{b}(b-t)^{\alpha -1}\left\vert t-(1-\lambda
)b-\lambda a\right\vert dt \\
&& \\
&=&\int\limits_{x}^{\lambda a+(1-\lambda )b}(b-t)^{\alpha -1}\left( \lambda
a+(1-\lambda )b-t\right) dt+\int\limits_{\lambda a+(1-\lambda
)b}^{b}(b-t)^{\alpha -1}\left( t-\lambda a-(1-\lambda )b\right) dt \\
&& \\
&=&\frac{2\lambda ^{\alpha +1}(b-a)^{\alpha +1}}{\alpha (\alpha +1)}+\frac{%
(b-x)^{\alpha }}{\alpha \left( \alpha +1\right) }\left[ \alpha (b-x)-\lambda
(b-a)\left( \alpha +1\right) \right] .
\end{eqnarray*}%
Using $J_{3}$ and $J_{4}$ in (\ref{10}), we obtain (\ref{9}).
\end{proof}
\end{theorem}

\begin{remark}
We note that in the special cases, if we take $\lambda =0$ in Theorem \ref%
{thm3}, then it reduces Theorem 4.1 proved by Anastassiou et. al. \cite%
{Anastassiou}. So, our results are generalizations of the corresponding
results of Anastassiou et. al. \cite{Anastassiou}.
\end{remark}


\begin{thebibliography}{99}
\bibitem{Anastassiou} G. Anastassiou, M.R. Hooshmandasl, A. Ghasemi and F.
Moftakharzadeh, \textit{Montgomery identities for fractional integrals and
related fractional inequalities}, J. Inequal. in Pure and Appl. Math, 10(4),
2009, Art. 97, 6 pp.

\bibitem{Belarbi} S. Belarbi and Z. Dahmani, \textit{On some new fractional
integral inequalities}, J. Inequal. in Pure and Appl. Math, 10(3), 2009,
Art. 97, 6 pp.

\bibitem{Cerone1} P. Cerone and S.S. Dragomir, \textit{Trapezoidal type
rules from an inequalities point of view}, Handbook of
Analytic-Computational Methods in Applied Mathematics, CRC Press N.Y. (2000).

\bibitem{Dahmani} Z. Dahmani, L. Tabharit and S. Taf, \textit{Some
fractional integral inequalities}, Nonlinear Science Letters A, 2(1), 2010,
p.155-160.

\bibitem{Dahmani1} Z. Dahmani, L. Tabharit and S. Taf, \textit{New
inequalities via Riemann-Liouville fractional integration}, J. Advance
Research Sci. Comput., 2(1), 2010, p.40-45.

\bibitem{Duo} J. Duoandikoetxea, \textit{A unified approach to several
inequalities involving functions and derivatives}, Czechoslovak Mathematical
Journal, 51 (126) (2001), 363--376.

\bibitem{gorenflo} R. Gorenflo, F. Mainardi, \textit{Fractionalcalculus:
integral and differentiable equations of fractional order}, Springer Verlag,
Wien, 1997, p.223-276.

\bibitem{mitrovich} D. S. Mitrinovic, J. E. Pecaric and A. M. Fink, \textit{%
Inequalities involving functions and their integrals and derivatives,}
Kluwer Academic Publishers, Dordrecht, 1991.

\bibitem{Dragomir} S.S. Dragomir and N. S. Barnett, \textit{An Ostrowski
type inequality for mappings whose second derivatives are bounded and
applications}, RGMIA Research Report Collection, V.U.T., 1(1999), 67-76.

\bibitem{Dragomir1} S.S. Dragomir, \textit{An Ostrowski type inequality for
convex functions}, Univ. Beograd. Publ. Elektrotehn. Fak. Ser. Mat. 16
(2005), 12--25.

\bibitem{Liu} Z. Liu, \textit{Some companions of an Ostrowski type
inequality and application}, J. Inequal. in Pure and Appl. Math, 10(2),
2009, Art. 52, 12 pp.

\bibitem{samko} S. G. Samko, A. A Kilbas, O. I. Marichev, \textit{\
Fractional Integrals and Derivatives Theory and Application}, Gordan and
Breach Science, New York, 1993.

\bibitem{sarikaya} M. Z. Sarikaya, \textit{On the Ostrowski type integral
inequality}, Acta Math. Univ. Comenianae, Vol. LXXIX, 1(2010), pp. 129-134.

\bibitem{sarikaya1} M. Z. Sarikaya, \textit{On the Ostrowski type integral
inequality for double integrals} , arXiv:1005.0454(submitted).

\bibitem{sarikaya2} M. Z. Sarikaya and H. Ogunmez, \textit{On the weighted
Ostrowski type integral inequality for double integrals},
arXiv:1005.0455(submitted).

\bibitem{Ostrowski} A. M. Ostrowski, \textit{\"{U}ber die absolutabweichung
einer differentiebaren funktion von ihrem integralmitelwert}, Comment. Math.
Helv. 10(1938), 226-227.
\end{thebibliography}
\end{document}